\documentclass[12pt]{article}
\usepackage{amsmath,amssymb,amsthm, comment}

\newtheorem{theorem}{Theorem}[section]
\newtheorem{lemma}[theorem]{Lemma}
\newtheorem{corollary}[theorem]{Corollary}

\newtheorem{proposition}[theorem]{Proposition}

\newtheoremstyle{named}{}{}{\itshape}{}{\bfseries}{.}{.5em}{\thmnote{#3's }#1}
\theoremstyle{named}

\newtheoremstyle{nnamed}{}{}{\itshape}{}{\bfseries}{.}{.5em}{\thmnote{#3' }#1}
\theoremstyle{nnamed}

\newcommand{\CD}{\mathcal{C}\mathcal{D}}

\def\barr{\begin{array}}
\def\earr{\end{array}}

\title{On the Chermak-Delgado lattice\\ of a finite group}
\author{Ryan McCulloch and Marius T\u arn\u auceanu}
\date{May 30, 2019}

\begin{document}

\maketitle

\begin{abstract}
By imposing conditions upon the index of a self-centralizing subgroup of a group, and upon the index of the center of the group, we are able to classify the Chermak-Delgado lattice of the group.  This is our main result.  We use this result to classify the Chermak-Delgado lattices of dicyclic groups and of metabelian $p$-groups of maximal class.
\end{abstract}

{\small
\noindent
{\bf MSC2000\,:} Primary 20D30; Secondary 20D60, 20D99.

\noindent
{\bf Key words\,:} Chermak-Delgado measure, Chermak-Delgado lattice, Chermak-Delgado subgroup, subgroup lattice, (generalized) dicyclic group, metabelian group.}

\section{Introduction}

Throughout this paper, $G$ will denote a finite group.  Denote by
\begin{equation}
m_G(H)=|H||C_G(H)|\nonumber
\end{equation}the \textit{Chermak-Delgado measure} of a subgroup $H$ of $G$ and let
\begin{equation}
m^*(G)={\rm max}\{m_G(H)\mid H\leq G\} \mbox{ and } {\cal CD}(G)=\{H\leq G\mid m_G(H)=m^*(G)\}.\nonumber
\end{equation} The set ${\cal CD}(G)$ forms a modular, self-dual sublattice of the lattice of subgroups of $G$,
which is called the \textit{Chermak-Delgado lattice} of $G$. It was first introduced by Chermak and Delgado \cite{5}, and revisited by Isaacs \cite{7}. In the last years there has been a growing interest in understanding this lattice (see e.g. \cite{1,2,3,4,6,8,9,10,13,14,15,18}). Recall two important properties of the Chermak-Delgado lattice that will be used in our paper:
\begin{itemize}
\item[$\cdot$] if $H\in {\cal CD}(G)$, then $C_G(H)\in {\cal CD}(G)$ and $C_G(C_G(H))=H$;
\item[$\cdot$] the minimum subgroup $M(G)$ of ${\cal CD}(G)$ (called the \textit{Chermak-Delgado subgroup} of $G$) is characteristic, abelian, and contains $Z(G)$.
\end{itemize}

For a positive integer $n\geq 1$, the dicyclic group of order $4n$, usually denoted by $Dic_{4n}$, is defined as
$$Dic_{4n}= \langle a,x\, |\, a^{2n}=1, x^2=a^n, a^x=a^{-1}\rangle.$$This has the following generalization: given an arbitrary abelian group $A$ of order $2n$, the generalized dicyclic group induced by $A$ is defined as
$$Dic_{4n}(A)=\langle A,x\, |\, x^4=1, x^2\in A\setminus \{1\}, a^x=a^{-1}, \forall\, a\in A\rangle.$$Obviously, we have $Dic_{4n}(\mathbb{Z}_{2n})=Dic_{4n}$. It is also easy to see that if $\exp(A)=2$, that is that $A$ is an elementary abelian $2$-group, then $Dic_{4n}(A)$ is abelian and consequently $\CD(Dic_{4n}(A))=\{Dic_{4n}(A)\}$. Note that if $\exp(A)\neq 2$, then
$$Z(Dic_{4n}(A))=\{a\in A\, |\,a^2=1\}\cong\frac{A}{A^2}\,,$$where $A^2=\{a^2\, |\, a\in A\}$.

A finite group $G$ is said to be \textit{metabelian} if the derived subgroup, $G'$, is abelian.  Equivalently, a finite group $G$ is metabelian if there exists an abelian normal subgroup $A$ of $G$ so that $G/A$ is abelian.

A finite $p$-group $G$ of order $p^n$ is said to be \textit{of maximal class} if the nilpotence class of $G$ is $n-1$.  The following results on $p$-groups will be useful to us.  Lemma 1.1 appears in (4.26), \cite{12}, II, Lemma 1.2 appears in Theorem 2.4 of \cite{16}, and Lemma 1.3 at end of \cite{14}.

\begin{lemma}
Any group of order $p^4$ contains an abelian subgroup of order $p^3$.
\end{lemma}

\begin{lemma}
Suppose $G$ is a $p$-group of order $p^n$ and $G$ is of maximal class.  Then $|Z(G)| = p$, $|G : G' | = p^2$, and for each $2 \leq i \leq n$, we have that $G_i$ is the unique normal subgroup of $G$ order $p^{n-i}$, where $G_i$ is the $i$th term in the lower central series for $G$.
\end{lemma}

\begin{lemma}
Let $G$ be a $p$-group of maximal class and of order $p^5$.  If $m^*(G) = p^6$, then $\CD(G) = \{G, T, G', A_1, \dots, A_p, Z(T), Z(G)\}$, where $|T| = p^4$, $|G'| = |A_i| = p^3$ for each $i$, $|Z(T)| = p^2$, and $G',A_1,\dots,A_p$ are all abelian and distinct.  Also, none of $A_1,\dots,A_p$ are normal in $G$.
\end{lemma}

\section{Main Results}

In this section we present a result which generalizes Proposition 7 in \cite{10}. It will be used both in Sections 3 and 4.  A subgroup $A$ of $G$ is said to be self-centralizing if $C_G(A) = A$.  An important observation used here is that if $T \in \CD(G)$, then since $Z(G) \leq T$, we have that $|G:T|$ divides $|G:Z(G)|$.  So by imposing conditions on $|G:Z(G)|$, we are able to obtain results about $\CD(G)$.

\begin{theorem}
Let $G$ be a finite group, let $p$ be a prime, and among the self-centralizing subgroups of $G$, let $A$ be one of maximum order.
\begin{itemize}
\item[{\rm 1)}] If $G$ is abelian, then $\CD(G) = \{G\}$.

\item[{\rm 2)}]  If $|G : Z(G)| = p^2$, then $\CD(G) = \{ Z(G), A, A_1, \dots, A_p, G \}$ is a quasi-antichain of width $p+1$, with each $A_i$ abelian.

\item[{\rm 3)}]  If $|G:A| = p$ and $|G : Z(G)| = p^i$ with $i > 2$, then $\CD(G) = \{ A \}$.

If $p$ is the smallest prime divisor of $|G|$ and $|G:A| = p$ and $|G : Z(G)| > p^2$, then $\CD(G) = \{ A \}$.

\item[{\rm 4)}] Suppose $|G:A| = p^2$.

\begin{itemize}
\item[{\rm a)}] If $|G:Z(G)| = p^3$, then $\CD(G) = \{ Z(G), G \}$.

 If $p$ is the smallest prime divisor of $|G|$ and $p^2 < |G:Z(G)| < p^4$, then $\CD(G) = \{ Z(G), G \}$.

\item[{\rm b)}] If $|G:Z(G)| = p^4$, then $$\CD(G) = \{ Z(G), Z(T_1),\dots,Z(T_n),A,A_1,\dots,A_m,T_1,\dots,T_n,G\}$$

where $T_1,\dots,T_n$ ($n \geq 0$) are all of the subgroups of index $p$ in $G$ with centers that have index $p^3$ in $G$, and $A_1,\dots,A_m$ ($m \geq 0$) are all of the subgroups (other than $A$) of index $p^2$ in $G$ with centralizers that have index $p^2$ in $G$.  Furthermore, if $n \geq 1$, then $m\geq p$.

\item[{\rm c)}] If $|G:Z(G)| = p^i$ with $i > 4$ and if $G$ possesses a subgroup, $T$, of index $p$ in $G$ with center that has index $p^3$ in $G$, then $\CD(G) = \{Z(T),A,A_1,\dots,A_p,T\}$ is a quasi-antichain of width $p+1$, with each $A_i$ abelian.

 If $p$ is the smallest prime divisor of $|G|$ and if $|G:Z(G)| > p^4$ and if $G$ possesses a subgroup, $T$, of index $p$ in $G$ with center that has index $p^3$ in $G$, then $\CD(G) = \{Z(T),A,A_1,\dots,A_p,T\}$ is a quasi-antichain of width $p+1$, with each $A_i$ abelian.

 \item[{\rm d)}] If $|G:Z(G)| = p^i$ with $i > 4$ and if $G$ does not possess a subgroup, $T$, of index $p$ in $G$ with center that has index $p^3$ in $G$, then $\CD(G) = \{A\}$.

 If $p$ is the smallest prime divisor of $|G|$ and if $|G:Z(G)| > p^4$ and if $G$ does not possess a subgroup, $T$, of index $p$ in $G$ with center that has index $p^3$ in $G$, then $\CD(G) = \{A\}$.
\end{itemize}
\end{itemize}
\end{theorem}

\begin{proof}
Item 1 is clear.

To see item 2, note first that $|G:G||G:Z(G)| = |G:Z(G)| = p^2 \leq |G:T||G:C_G(T)|$ for any subgroup $T$ of $G$ with $Z(G) \leq T$, and so $G,Z(G) \in \CD(G)$.  Note that $G/Z(G) \cong C_p \times C_p$, and so there are exactly $p+1$ subgroups $H$ with $Z(G) < H < G$.  For any such $H$, we have that $H = \langle Z(G), x \rangle$ for some $x \in H \setminus Z(G)$, and thus $H$ is abelian.  It follows that $\CD(G) = \{ Z(G), A, A_1, \dots, A_p, G \}$ is a quasi-antichain of width $p+1$, with each $A_i$ abelian.

We prove both parts of item 3 simultaneously.  Note that $|G:G||G:Z(G)| = |G:Z(G)| > p^2 = |G:A||G:A|$, and so $G,Z(G) \notin \CD(G)$.  Note that $|G:A||G:A| = p^2 \leq |G:T||G:C_G(T)|$ for any subgroup $T$ of $G$ with $Z(G) \leq T$, and so $A \in \CD(G)$.  Now $A$ is a maximal subgroup of $G$ and $G \notin \CD(G)$, and so $A$ is the largest member of $\CD(G)$.   But $A$ is self-centralizing and therefore $A$ is also the least member of $\CD(G)$. Hence $\CD(G)=\{A\}$.

We prove both parts of item 4.a simultaneously.  Note that $|G:A||G:A| = p^4 > |G:Z(G)| = |G:G||G:Z(G)|$, and so $A \notin \CD(G)$.  Since among the self-centralizing subgroups of $G$, $A$ is one of maximum order, it follows that $\CD(G)$ does not contain any self-centralizing subgroups.  Among the abelian subgroups of $\CD(G)$, let $T$ be a maximal one.  If $T \neq Z(G)$, then we have that $T < C_G(T) < G$, and furthermore we have that $|C_G(T) : T|$ cannot be a prime.  This is true because if $|C_G(T) : T|$ is a prime, then $C_G(T)/T$ is cyclic, and hence $C_G(T)$ is abelian, and so $T$ would not be maximal among the abelian subgroups in $\CD(G)$.  It follows that $|G:T||G:C_G(T)| \geq p^4 > |G:G||G:Z(G)|$, a contradiction.  Hence $T=Z(G)$ is the only abelian subgroup in $\CD(G)$.  If $H \in \CD(G)$ is a subgroup of index $p$ in $G$ with $C_G(H)$ a subgroup of index $p^2$ in $G$, then $|C_G(H) : Z(G)| < p^2$, and hence is prime.  But then $C_G(H)$ is abelian, a contradiction.  It follows that $\CD(G) = \{ Z(G), G \}$.

To see item 4.b, note that $p^4 = |G:G||G:Z(G)| = |G:A||G:A|$.   Since among the self-centralizing subgroups of $G$, $A$ is one of maximum order, we have that any other self-centralizing subgroup $A'$ of $\CD(G)$ will have that $p^4 = |G:A'||G:A'|$.  If $T$ is an abelian subgroup of $\CD(G)$ with $T < C_G(T) < G$ and $C_G(T)$ nonabelian, then as we saw in the proof of item 4.a, we have that $|G:T||G:C_G(T)| \geq p^4$.  It follows then that $G,Z(G),A \in \CD(G)$.  And hence $$\{ Z(G), Z(T_1),\dots,Z(T_n),A,A_1,\dots,A_m,T_1,\dots,T_n,G\} \subseteq \CD(G)$$ where $T_1,\dots,T_n$ ($n \geq 0$) are all of the subgroups of index $p$ in $G$ with centers that have index $p^3$ in $G$, and $A_1,\dots,A_m$ ($m \geq 0$) are all of the subgroups (other than $A$) of index $p^2$ in $G$ with centralizers that have index $p^2$ in $G$.

We also have that $$\CD(G) \subseteq \{ Z(G), Z(T_1),\dots,Z(T_n),A,A_1,\dots,A_m,T_1,\dots,T_n,G\}.$$

This is true because, otherwise, if some $H \in \CD(G)$ has that $|G:H|=p$ and $|G:C_G(H)|=p^3$, but $Z(H) < C_G(H)$, then since $Z(H) = H \cap C_G(H) \in \CD(G)$, it would follow that $Z(H) = Z(G)$ and thus $H=G$, a contradiction.

If $n \geq 1$, then for some $T \in \{T_1, \dots, T_n\}$, we have that $Z(T) < A < T$ in $\CD(G)$.  This is true because, otherwise, we would have $Z(G) < A < G$ and $Z(G) < Z(T) < T < G$ as maximal chains in $\CD(G)$, contradicting the modularity of $\CD(G)$.  We can apply item 2 in Theorem 2.1 to the group $T$, to conclude that there are $p$ additional self-centralizing groups in $\CD(T)$, and since all of these are self-centralizing, they are also all in $\CD(G)$.

We prove both parts of item 4.c simultaneously.  Note that $|G:G||G:Z(G)| > p^4 = |G:A||G:A|$, and so $G,Z(G) \notin \CD(G)$.  As we saw before, if $S$ is an abelian subgroup of $\CD(G)$ with $S < C_G(S) < G$ and $C_G(S)$ nonabelian, then $|G:S||G:C_G(S)| \geq p^4$.  It follows that $A,T,Z(T) \in \CD(G)$, and since $T$ is maximal in $G$, we have that $T$ is the largest member of $\CD(G)$.  And so applying Theorem 2.1 item 2 to the group $T$, we obtain that $\CD(G) = \{Z(T),A,A_1,\dots,A_p,T\}$ is a quasi-antichain of width $p+1$, with each $A_i$ abelian.

We prove both parts of item 4.d simultaneously.  As we saw in the proof of item 4.c, $G,Z(G) \notin \CD(G)$ and $A \in \CD(G)$.  But since there is no subgroup of index $p$ in $\CD(G)$, we have that $A$ is the largest member of $\CD(G)$, and since $A$ is self-centralizing, $A=C_G(A)$ is the least member of $\CD(G)$, and we conclude that $\CD(G) = \{ A \}$.
\end{proof}

Items 3, 4.c, and 4.d in Theorem 2.1 have interesting consequences due to the uniqueness of the largest member of the Chermak-Delgado lattice.  We collect them below as a corollary.

\begin{corollary}
Let $G$ be a finite group, let $p$ be a prime, and among the self-centralizing subgroups of $G$, let $A$ be one of maximum order.
\begin{itemize}
\item[{\rm 1)}]  If $|G:A| = p$ and $|G : Z(G)| = p^i$ with $i > 2$, then $A$ is unique.

If $p$ is the smallest prime divisor of $|G|$ and $|G:A| = p$ and $|G : Z(G)| > p^2$, then $A$ is unique.

\item[{\rm 2)}] If $|G:A| = p^2$ and $|G:Z(G)| = p^i$ with $i > 4$ and if $G$ possesses a subgroup, $T$, of index $p$ in $G$ with center that has index $p^3$ in $G$, then $T$ is unique.

 If $p$ is the smallest prime divisor of $|G|$ and if $|G:A| = p^2$ and $|G:Z(G)| > p^4$ and if $G$ possesses a subgroup, $T$, of index $p$ in $G$ with center that has index $p^3$ in $G$, then $T$ is unique.

\item[{\rm 3)}] If $|G:A| = p^2$ and $|G:Z(G)| = p^i$ with $i > 4$ and if $G$ does not possess a subgroup, $T$, of index $p$ in $G$ with center that has index $p^3$ in $G$, then $A$ is unique.

 If $p$ is the smallest prime divisor of $|G|$ and if $|G:A| = p^2$ and $|G:Z(G)| > p^4$ and if $G$ does not possess a subgroup, $T$, of index $p$ in $G$ with center that has index $p^3$ in $G$, then $A$ is unique.
\end{itemize}
\end{corollary}

Examples for items 4.a and 4.b in Theorem 2.1 that are chains of length 1 and length 2 (and so in the case of item 4.b, $n=0$ and $m=0$), can be found in \cite{3}, see Corollary 2.2, Proposition 2.3, Corollary 2.5, and Proposition 2.6 there.

An example for item 4.b in Theorem 2.1 with $n=1$ and $m=p$ will be given in Section 4.

Another example for item 4.b in Theorem 2.1 is found by taking $G$ to be one of the two extraspecial $p$-groups of order $p^5$.  $|Z(G)| = p$ (i.e. $|G:Z(G)| = p^4$) and $G$ possesses a maximal self-centralizing elementary
abelian subgroup of order $p^3$ (i.e. of index $p^2$ in $G$). The Chermak-Delgado lattice of extraspecial $p$-groups is known, see Example 2.8 in \cite{6}.  And so we have that $\CD(G)$  is isomorphic to the subgroup lattice of an elementary
abelian $p$-group of order $p^4$.  Using the notation of item 4.b in Theorem 2.1, we have that $n=p^3+p^2+p+1$ and $m=(p^2+1)(p^2+p+1) - 1$.

Further examples of groups and CD lattices arising from Theorem 2.1 will be explored in the next sections.

\section{Dicyclic groups}

As a consequence of items 2 and 3 in Theorem 2.1, we classify the Chermak-Delgado lattice of $Dic_{4n}(A)$, where $A$ is an abelian group of order $2n$.

Note that $\CD(Dic_{4n}(A)) = \{A\}$ if and only if $|Dic_{4n}(A) : Z(Dic_{4n}(A))| > 4$,
which is equivalent to $|A^2| > 2$. If $A$ is not an elementary abelian
$2$-group, this means that $A$ is not a direct product of an elementary
abelian $2$-group and $\mathbb{Z}_4$.

\begin{theorem}
    Let $A$ be an abelian group of order $2n$ and $Dic_{4n}(A)$ be the generalized dicyclic group induced by $A$. Then we have:
\begin{itemize}
\item[{\rm a)}] If $A$ is not of type $\mathbb{Z}_2^m\times\mathbb{Z}_4$ with $m\in\mathbb{N}$, then $\CD(Dic_{4n}(A))$ is a chain of length $0$, namely $\CD(Dic_{4n}(A))=\{Dic_{4n}(A)\}$ for $\exp(A)=2$, and $\CD(Dic_{4n}(A))=\{A\}$ for $\exp(A)\neq 2$.
\item[{\rm b)}] If $A$ is of type $\mathbb{Z}_2^m\times\mathbb{Z}_4$ with $m\in\mathbb{N}$, then $\CD(Dic_{4n}(A))$ is a quasi-antichain of width $3$, namely the lattice interval between $Z(Dic_{4n}(A))$ and $Dic_{4n}(A)$.
\end{itemize}
\end{theorem}

As a corollary, we describe the Chermak-Delgado lattice of $Dic_{4n}$.

\begin{corollary}
    Under the previous notation, we have:
\begin{itemize}
\item[{\rm a)}] If $n\neq 2$, then $\CD(Dic_{4n})$ is a chain of length $0$, namely $\CD(Dic_{4n})=\{Dic_{4n}\}$ for $n=1$, and $\CD(Dic_{4n})=\{\langle a\rangle\}$ for $n\geq 3$.
\item[{\rm b)}] If $n=2$, then $\CD(Dic_{4n})$ is a quasi-antichain of width $3$, namely the lattice interval between $Z(Dic_{4n})$ and $Dic_{4n}$.
\end{itemize}
\end{corollary}

The following appears in \cite{10}.

\begin{theorem}
    Let $G$ be a finite group which can be written as $G=AB$,
    where $A$ and $B$ are abelian subgroups of relatively prime
    orders and $A$ is normal. Then
    \begin{equation}
    m(G)=|A|^2|C_B(A)|^2 \mbox{ and }   {\cal CD}(G)=\{AC_B(A)\}.\nonumber
    \end{equation}
\end{theorem}

\begin{corollary}
Suppose $A$ is a finite abelian group.  There exists a non-abelian finite group $G$ so that $\CD(G) = \{ A \}$ if and only if $A \neq 1$, $A \ncong \mathbb{Z}_2$, $A \ncong \mathbb{Z}_4$, and $A \ncong\mathbb{Z}_2 \times \mathbb{Z}_4$.
\end{corollary}

\begin{proof}
$\rightarrow$ Suppose $\CD(G) = \{ A \}$ and $G$ is non-abelian.   If $A =1$, then $G = 1$ is abelian, a contradiction.  Since $A$ is self-centralizing and normal in $G$, we have that $G/A$ embeds into $Aut(A)$.   If $A \cong \mathbb{Z}_2$, then $G/A$ embeds into $Aut(A) = 1$, which contradicts $G$ being non-abelian.   If $A \cong \mathbb{Z}_4$, then $G/A$ embeds into $Aut(A)$ which has order $2$, and so $|G|=8$ and $CD(G)$ would not be a chain of length zero by item 2 in Theorem 2.1.  Suppose $A \cong \mathbb{Z}_2 \times \mathbb{Z}_4$.  Then $G/A$ embeds into $Aut(A)$ which is a dihedral group of size $8$.  And so $G$ is a $2$-group.  Now, $m^*(G) = 64$, and since $G \notin \CD(G)$ and $G$ has a nontrivial center, it follows that $|G| = 16$ and $|Z(G)| = 2$.  Note, however, that given $x \in Aut(A)$ of order 2, $x$ must fix at least four elements of $A$.  This is due to the structure of $Aut(A)$ when $A \cong \mathbb{Z}_2 \times \mathbb{Z}_4$.  And so $|Z(G)| > 2$, a contradiction.

$\leftarrow$  If $|A|$ is odd, then let $G = A \rtimes \langle x \rangle$ where $x$ inverts each element of $A$.  Then by item 3 in Theorem 2.1, $\CD(G) = \{ A \}$.  If $|A|$ is even, and $A$ is not of type $\mathbb{Z}_2^m\times\mathbb{Z}_4$ with $m \geq 0$, and $\exp(A) \neq 2$, then $G = Dic_{4n}(A)$ has that $\CD(G) = \{ A \}$ by Theorem 3.1.  If $A \cong \mathbb{Z}_2^m$ with $m \geq 2$ or $A \cong \mathbb{Z}_2^m \times \mathbb{Z}_4$ with $m \geq 2$, then there exists $x \in Aut(A)$ of order 3.  Let $G = A \rtimes \langle x \rangle$.  One can apply Theorem 3.3 to get that $\CD(G) = \{ A \}$.
\end{proof}

Finally, we present a proposition also applicable to describing $\CD(Dic_{4n})$.

\begin{proposition}
Let $G$ be a finite group.  If $G = HZ(G)$ for some subgroup $H$ of $G$, then
$\CD(G)$ and $\CD(H)$ are lattice isomorphic with $\CD(G) = \{ XZ(G) \,\, | \,\, X
\in \CD(H) \}$.
\end{proposition}

\begin{proof}
Since $G=HZ(G)$, we have that $G/Z(G) \cong H/(Z(G) \cap H) = H/Z(H)$ since $Z(H)
\leq Z(G)$.  This natural isomorphism between $G/Z(G)$ and $H/Z(H)$ induces a
lattice isomorphism between the subgroup intervals $[G:Z(G)]$ and $[H:Z(H)]$, and
this lattice isomorphism restricts to a lattice isomorphism between $\CD(G)$ and
$\CD(H)$.  The result follows.
\end{proof}

\section{Metabelian $p$-groups of maximal class}

By applying Lemmas 1.1, 1.2, 1.3 and Theorem 2.1, we obtain a classification of the Chermak-Delgado lattices of a metabelian $p$-groups of maximal class.

\begin{theorem}
Suppose $G$ is a metabelian $p$-group of order $p^n$ and of maximal class.

\begin{itemize}
\item [{\rm 1 )}] If $n=3$, then $\CD(G) = \{Z(G),A,A_1,\dots,A_p,G\}$ is a quasi-antichain of width $p+1$, with each $A_i$ abelian.

\item [{\rm 2 )}] If $G$ possesses an abelian subgroup $A$ so that $|G:A| =p$ and $n > 3$, then $\CD(G) = \{A\}$.

\item [{\rm 3 )}] Suppose that $G$ does not possess an abelian subgroup $A$ so that $|G:A| =p$. Then we have that $n > 4$, and, furthermore:
\begin{itemize}
\item[{\rm a)}] If $|G| = p^5$, then  $\CD(G) = \{Z(G), Z(T), G', A_1, \dots, A_p, T, G\}$, where $|T| = p^4$, $|G'| = |A_i| = p^3$ for each $i$, $|Z(T)| = p^2$, and $G',A_1,\dots,A_p$ are all abelian and distinct.  Also, none of $A_1,\dots,A_p$ are normal in $G$.
\item[{\rm b)}] If $|G| > p^5$ and $G$ possesses a subgroup $T$ of index $p$ so that $|T:Z(T)| = p^2$, then $\CD(G) = \{ Z(T), G',A_1,\dots,A_p, T \}$, where $|T| = p^{n-1}$, $|G'| = |A_i| = p^{n-2}$ for each $i$, $|Z(T)| = p^{n-3}$, and $G',A_1,\dots,A_p$ are all abelian and distinct.  Also, none of $A_1,\dots,A_p$ are normal in $G$.
\item[{\rm c)}] If $|G| > p^5$ and $G$ does not possess a subgroup $T$ of index $p$ so that $|T:Z(T)| = p^2$, then $\CD(G) = \{ G' \}$.
\end{itemize}
\end{itemize}
\end{theorem}

\begin{proof}
Items 1 and 2 follow directly from items 2 and 3 in Theorem 2.1.  The fact that $n > 4$ in item 3 follows from Lemma 1.1.  Most of Theorem 4.1 part 3 is a straight forward application of Lemma 1.2 and Theorem 2.1 taking $A = G'$.  In item 3.a, the fact that there is a unique subgroup, $T$, of index $p$ in $G$ in $\CD(G)$ is a nontrivial result that follows from Lemma 1.3 which in turn relies on Lemma 4.5.4 in \cite{14}.  In items 3.a and 3.b, the fact that none of $A_1,\dots,A_p$ are normal in $G$ follows from Lemma 1.2.
\end{proof}

Using Theorem 4.1, we are able to give some necessary and sufficient conditions under which the Chermak-Delgado lattice of a metabelian $p$-group of maximal class is a chain of length $0$.

\begin{corollary}
   Let $G$ be a metabelian $p$-group of maximal class.  Then $\CD(G)$ is a chain of length $0$ if and only if either $G$ possesses an abelian subgroup of index $p$ and $|G: Z(G)|\neq p^2$, or $G$ does not possess abelian subgroups of index $p$, $|G|>p^5$, and $|M:Z(M)|> p^2$ for any maximal subgroup $M$ of $G$.
\end{corollary}

We provide now a few remarks.  The $2$-groups of maximal class fall into $3$ categories: the dihedral groups, the semidihedral groups, and the generalized quaternion groups (see e.g. Theorem 4.1 of \cite{12}, II), and each of these possess a cyclic subgroup of index $2$. Thus, the CD lattices of these groups fall under the umbrella of items 2 and 3 in Theorem 2.1.

Xue, Lv, and Chen provide sufficient conditions for when $p$-groups of maximal class have CD lattices that fall under the umbrella of item 3 in Theorem 2.1 (see Theorem 2.9 of \cite{17}):

\begin{proposition}
Let $G$ be a finite $p$-group of order $p^n$ and of maximal class, where $p\geq 5$ and $n>2p$. If $C_G(H)=Z(H)$ for every non-abelian subgroup
$H$ of $G$, then $G$ possesses an abelian subgroup of index $p$.
\end{proposition}

Note that the finite non-abelian groups satisfying the condition in Proposition 4.3 are called \textit{CGZ-groups}.  Xue, Lv, and Chen show that if $G$ is a $p$-group of maximal class and $G$ is a CGZ-group, then $G$ is metabelian (see Proposition 2.6 of \cite{17}).
\bigskip

We end with examples of metabelian $p$-groups of maximal class for each of the three cases in item 3 of Theorem 4.1.  For the case a) such an example is presented in \cite{17}:
\begin{equation}
G{=}\langle a,b,c,d \mid a^{p^2}{=}b^p{=}c^p{=}d^p{=}1, [b,a]{=}c, [c,a]{=}d, [c,b]{=}[d,b]{=}a^p, [d,a]{=}1\rangle,\nonumber
\end{equation}where $p\geq 5$.
\bigskip

For case b) we present $SmallGroup(3^6,99)$ from GAP:

\begin{equation}
G{=}\langle a,b,c,d,e,f \mid a^3 {=}e^3{=}f^3{=}1, b^3 {=} c^2d, c^3 {=} e^2f, d^3 {=} f^2, [b,a]{=}c, [c,a]{=}d, [d,a]{=}e,
\nonumber
\end{equation}
\begin{equation}
[e,a]{=}[c,b]{=}f, [f,a]{=}[d,b]{=} [e,b]{=} [f,b]{=} [d,c]{=} [e,c]{=}[f,c]{=} [e,d]{=} [f,d]{=} [f,e]{=}1 \rangle.\nonumber
\end{equation}
\bigskip

For case c) we present $SmallGroup(5^6,651)$ from GAP:

\begin{equation}
G{=}\langle a,b,c,d,e,f \mid a^5 {=}c^5 {=}d^5 {=}e^5{=}f^5{=}1, b^5 {=} f^3, [b,a]{=}c,[c,b]{=}d, [c,a]{=}[d,b]{=}e,
\nonumber
\end{equation}
\begin{equation}
[d,a]{=}[e,b]{=}f, [e,a]{=} [f,a]{=} [f,b]{=} [d,c]{=} [e,c]{=}[f,c]{=} [e,d]{=} [f,d]{=} [f,e]{=}1 \rangle.\nonumber
\end{equation}
\bigskip

\noindent{\bf Acknowledgments.} The authors are grateful to the reviewer for its remarks which improve the previous version of the paper.

\vspace*{3ex}
\small

\begin{minipage}[t]{7cm}
Ryan McCulloch\\
Assistant Professor of Mathematics\\
University of Bridgeport\\
Bridgeport, CT 06604\\
e-mail: {\tt rmccullo@bridgeport.edu}
\end{minipage}
\hfill
\begin{minipage}[t]{7cm}
Marius T\u arn\u auceanu \\
Faculty of  Mathematics \\
``Al.I. Cuza'' University \\
Ia\c si, Romania \\
e-mail: {\tt tarnauc@uaic.ro}
\end{minipage}

\end{document}